\documentclass{amsart}
\usepackage{amssymb,amsmath,amscd,amsthm}
\usepackage{amsfonts,epsfig,latexsym,graphicx,amssymb}
\usepackage{color}
\usepackage{enumerate}
\usepackage{todonotes}

\usepackage{hyperref}

\date{\today}

\newcommand{\Z}{{\mathbb Z}}
\newcommand{\R}{{\mathbb R}}

\newcommand{\tg}{\tilde g}
\newcommand{\tG}{\tilde G}

\DeclareMathOperator\Homeo{Homeo}

\newtheorem{theorem}{Theorem}[section]
\newtheorem*{theorem*}{Theorem}
\newtheorem{lemma}[theorem]{Lemma}
\newtheorem{prop}[theorem]{Proposition}
\newtheorem{coro}[theorem]{Corollary}

\theoremstyle{definition}
\newtheorem{remark}[theorem]{Remark}

\title[Log H\"older continuity of the rotation number]{Log H\"older continuity of the rotation number}

\author[A.\ Gorodetski]{Anton Gorodetski}

\address{Department of Mathematics, University of California, Irvine, CA~92697, USA}

\email{asgor@uci.edu}

\thanks{A.\ G.\ was supported in part by NSF grant DMS--2247966.}

\author[V. Kleptsyn]{Victor Kleptsyn}

\address{CNRS, Institute of Mathematical Research of Rennes, IRMAR, UMR 6625 du CNRS}

\email{victor.kleptsyn@univ-rennes1.fr}

\thanks{V.K. was supported in part by ANR Gromeov (ANR-19-CE40-0007) and by Centre Henri Lebesgue (ANR-11-LABX-0020-01)}

\begin{document}

\maketitle

\

\begin{abstract}
We consider one-parameter families of smooth circle cocycles over an ergodic transformation in the base, and show that their rotation numbers must be log-H\"older regular with respect to the parameter. As an immediate application, we get a dynamical proof of 1D version of the Craig-Simon theorem that establishes 
that the integrated density of states of an ergodic Schr\"odinger operator must be log-H\"older.
\end{abstract}

\section{Introduction}\label{s.intro}

The rotation number of a homeomorphism or diffeomorphism of the circle was introduced by H.\,Poincare~\cite{Po} in 1885. Since then this notion was extensively studied, e.g. see \cite[Sections 11, 12]{KH} for a modern exposition of the main results. In particular, in the case when the circle diffeomorphism depends on a parameter $a\in \mathbb{R}^1$, properties of the rotation number $\rho$ as a function of the parameter, $\rho=\rho(a)$,  is a classical topic in dynamical systems.

In many cases the graph of function $\rho(a)$ turns out to be a ``devil's staircase'', with many fascinating properties. It was shown that, under suitable conditions, the function $\rho(a)$ must be continuous, but in general not Lipschitz \cite{Arn, Her1}, 
generically of bounded variation \cite{Bru}, and H\"older continuous \cite{Gr}. H\"older continuity of $\rho(a)$ for families of diffeomorphisms 
 with a critical point was established in \cite{GrS}, see also \cite{Kh}.  Increasingly refined results on the properties of  $\rho(a)$ appear up to this day \cite{Mat}.

Here we consider the rotation number as a function of a parameter not for one circle map, but for a cocycle over an ergodic transformation in a base, with smooth circle maps on the fibers; we provide the formal setting in Section \ref{s.2}. It turns out that in this case the rotation number does not have to be a H\"older continuous function of the parameter, but must be log-H\"older, see Theorem \ref{t.main} below for the formal statement.

 Our initial motivation for writing this paper came from an attempt to understand a dynamical meaning of the famous Craig-Simon Theorem from spectral theory \cite{CS, cs}, and Theorem \ref{t.main} can be interpreted as its non-linear version. The Craig-Simon Theorem  claims that the integrated density of states of an ergodic family of discrete Sch\"odinger operators must be  log-H\"older continuous, and can be reformulated as a statement about the rotation number of a projectivization of the Schr\"odinger cocycle that depends on energy as a parameter. In this way, Theorem \ref{t.main} provides a 1D version of the Craig-Simon Theorem as an immediate corollary. The original proof of the Craig-Simon Theorem in \cite{CS, cs} used very different arguments, and was based on the so called Thouless formula. A beautiful dynamical version of the Thouless formula was derived recently in \cite{BCDFK}. In particular, under suitable technical conditions, it implies log-H\"older continuity of the rotation number for general affine one-parameter families of projective cocycles, not just Schr\"odinder cocycles, see \cite[Proposition 5.1]{BCDFK}, which is also a partial case of Theorem \ref{t.main}.

In Section \ref{s.2} below we provide the setting and formulate and prove the main result, Theorem \ref{t.main}. In Section \ref{s.3} we give the background from spectral theory needed to formulate the Craig-Simon Theorem, and explain its relation to Theorem \ref{t.main}. Also, referring to known results in spectral theory, we notice that Theorem \ref{t.main} is essentially optimal.

\section{Preliminaries and Main Result}\label{s.2}

\subsection{Preliminaries}
Suppose that $\frak{M}$ is a compact metric space, $\sigma:\frak{M}\to \frak{M}$ is a homeomorphism, and $\mu$ is an ergodic invariant Borel probability measure supported on $\frak{M}$. Assume also that we are given a continuous map $g_{\cdot}:\frak{M}\to {\Homeo}^+(S^1)$, where by $\Homeo^+(S^1)$ we denote the space of orientation-preserving homeomorphisms of the circle, where $S^1=\mathbb{R}/\mathbb{Z}$ denotes the circle. Then, one can consider an associated skew product
$$
F:(\omega, x)\mapsto (\sigma \omega, g_\omega (x)).
$$

Next, let us choose for every $\omega\in\frak{M}$ a lift $\tilde g_\omega:\mathbb{R}\to \mathbb{R}$ of the map $g_\omega\in Homeo^+(S^1)$,
$$
g_\omega(\pi(x))=\pi(\tilde g_\omega(x)),
$$
where $\pi:\mathbb{R}^1\to S^1$ 
 is a natural covering map, in such a way that $\{\tilde g_\omega(0)\}$ is a bounded measurable (in $\omega$) function (e.g. one can require $g_\omega(0)\in [0,1)$ for all $\omega\in \frak{M}$). We then can consider the associated lift of the skew product:
\begin{equation}\label{e.cocycle}
\tilde F:(\omega, x)\mapsto (\sigma \omega, \tilde g_\omega (x)).
\end{equation}

Finally, let $G_{m,\omega}$ and $\tilde G_{m,\omega}$ be the length $m$ fiberwise compositions associated to these skew products:
$$
F^m (\omega,x)=(\sigma^m \omega, G_{m,\omega}(x)), \quad \tilde F^m (\omega,x)=(\sigma^m \omega, \tilde G_{m,\omega}(x)),
$$
so that for $m>0$ we have
$$
\tilde G_{m,\omega}=\tilde g_{\sigma^{m-1}\omega}\circ \ldots \circ \tilde g_{\sigma\omega}\circ\tilde g_{\omega}.
$$

The following statement is well known (e.g. see \cite[Section 5]{Her}, \cite{R2}, or \cite[Appendix A]{GK}),
\begin{prop}\label{p.rotnum}
There exists a number $\rho\in \mathbb{R}$ such that for $\mu$-a.e. $\omega\in \frak{M}$ and every $x\in \mathbb{R}$ the limit
\begin{equation}\label{eq:tG}
\lim_{n\to \infty}\frac{1}{n} (\tilde G_{n,\omega} (x)-x)
\end{equation}
exists and is equal to $\rho$.
\end{prop}

 The number $\rho$ from Proposition \ref{p.rotnum} is called \emph{the rotation number}.
Notice that the rotation number $\rho$ depends on the choice of lifts $\tilde g_\omega$.

\begin{remark}
 It can happen that the lifts $\{\tilde g_\omega\}$ cannot be taken continuous in~$\omega$. At the same time, in the case when $\{g_\omega\}$ are projectivizations of the transfer matrices of a Schr\"odinger cocycle defined by a continuous potential, the lifts $\{\tilde g_\omega\}$ can always be chosen continuously in $\omega$ (since any Schr\"odinger cocycle is homotopic to a constant one).
 \end{remark}

\begin{remark}
 Some of the assumptions in Proposition \ref{p.rotnum} can be essentially relaxed. For example, one can start with a probability space $(\frak{M}, \mu)$ and a measure preserving transformation $\sigma$ instead on a measure preserving homeomorphism of a compact metric space. 
  To keep the presentation more transparent, we are not trying to give the statements in the most general form.
\end{remark}

Let us now consider the dependence of the rotation number on a parameter. Namely, assume now that we are given a continuous
family $g_{\cdot,\cdot}: J\times \frak{M} \to \Homeo_+(S^1)$ of maps as above; here $J\subseteq \mathbb{R}^1$ is a closed interval of parameters.  Then, we can consider their lifts $\tg_{a,\omega}:\R\to\R$ to be chosen continuously in parameter $a\in J$. The corresponding skew products $F_a$ and $\tilde F_a$ as well as the fiberwise compositions
$G_{n,a,\omega}$ and $\tilde G_{n,a,\omega}$ then can be defined in the same way as before.

%

 An important note is  that the \emph{increments} of the images $\tG_{n,a',\omega}(x)-\tG_{n,a,\omega}(x)$ do not depend on a particular choice of lifts $\tg_{a,\omega}$. Moreover, this increment is  \emph{continuous} in $\omega$ and $x$ (and in fact is a well-defined function of the point~$x$ on the circle, not only on the real line). Also, dividing by $n$ and passing to the limit, one gets that the difference of the corresponding rotation numbers $\rho(a')-\rho(a)$ does not depend on the choice of lifts~$\tg_{a,\omega}$, thus getting the following important note.

\begin{remark}
Even though the rotation number $\rho$ depends on a particular choice of the lifts $\tg_{a,\omega}$, the differences of rotation numbers $\rho(a')-\rho(a)$ do not. In particular, different choice of lifts $\tilde g_{a, \omega}$ leads to a shift of the rotation number $\rho(a)$ by a constant, and intervals in the space of parameters where  $\rho$ is constant are independent of the choice of the lifts.
\end{remark}

\subsection{Main result}
Here is the main result of this paper:

\begin{theorem}\label{t.main}
Assume that the cocycle (\ref{e.cocycle}) smoothly depends on the parameter $a\in J$ and satisfies the following:

\vspace{3pt}

1) The range of parameters is a closed interval $J\subset \mathbb{R}$, and for some uniform (in $x\in \mathbb{R}$, $\omega\in \frak{M}, a\in J$) constant $C>0$ one has
$$
\left|\frac{\partial \tg_{a, \omega}(x)}{\partial a}\right|\le C;
$$

\vspace{3pt}

2) If we set $M_\omega=\max\left\{2, \max_{x\in \mathbb{R}, \, a\in J}\left|\frac{\partial\tg_{a,\omega}(x)}{\partial x}\right|\right\}$, then
$$
\int_{\frak{M}}\log M_\omega\, d\mu(\omega)<\infty.
$$
Then the rotation number as a function of the parameter is log-H\"older continuous, i.e. there exists $R>0$ such that for any $a, a'\in J$ with $|a-a'|\le 1/2$  we have
\begin{equation}\label{e.logH}
|\rho(a')-\rho(a)|\le R\left(\log|a'-a|^{-1}\right)^{-1}
\end{equation}
\end{theorem}

\begin{remark}
The restriction on the modulus of continuity given by (\ref{e.logH}) is optimal, and cannot be improved without restriction of the class of cocycles under consideration. See the last paragraph of Section \ref{s.3} for details.
\end{remark}
\begin{remark}
Instead of assuming smoothness of the maps $\tg_{a,\omega}(x)$ one can assume that they are Lipschitz, and replace $\left|\frac{d\tg_{a,\omega}(x)}{dx}\right|$ by the Lipschitz constant in the definition of $M_\omega$.
\end{remark}

\subsection{Proof of the main result}

\begin{proof}
Fix $a', a\in J$, and set $\delta=C|a'-a|$. It suffices to show that (\ref{e.logH}) holds for all $a, a'\in J$ that are sufficiently close to each other. Therefore, without loss of generality we can assume
\begin{equation}\label{e.aaC}
\delta<0.1\ \ \text{and}\ \ |a'-a|<C.
\end{equation}

Take any initial point $x_0=x_0'\in \R$, for instance, $x_0=x_0'=0$. Consider the sequences of its iterates associated to some $\omega\in \Omega$ and two different parameter values $a,a'\in J$, for $n\ge 1$
\[
x_{n}=\tG_{n, a, \omega}(x_0)=\tg_{a, \sigma^{n-1}(\omega)}(x_{n-1}),
\]
\[
x'_{n}=\tG_{n, a', \omega}(x_0)=\tg_{a', \sigma^{n-1}(\omega)}(x'_{n-1})
\]
(to simplify the notation, we do not indicate the dependence on~$\omega$ explicitly).

Then for $\mu$-a.e. $\omega\in \mathfrak{M}$ we have
\begin{equation}\label{e.rho}
\rho(a')-\rho(a)=\lim_{n\to \infty}\frac{1}{n}(x'_n - x_n).
\end{equation}

To prove~\eqref{e.logH}, we can assume without loss of generality that $\rho(a')>\rho(a)$, as the other case differs only by interchanging~$a$ and~$a'$.
The following lemma compares the evolution of distances between the two orbits (considered for the same $\omega$) :
\begin{lemma}\label{l.one}
For any $n=1,2,\dots$ and any integer $j$ the following holds:
\begin{equation}\label{eq:d-evolution}
x'_n-x_n - j  \le \delta + M_{\sigma^{n-1}\omega} \cdot \max (0,x'_{n-1}-x_{n-1} - j)
\end{equation}
\end{lemma}
\begin{proof}
Consider the point $y:=\tg_{a, \sigma^{n-1}(\omega)}(x'_{n-1})$. Applying the Lagrange theorem, due to the choice of $C$ we then have
\[
|x'_n-y| = \left| \tg_{a', \sigma^{n-1}(\omega)}(x'_{n-1}) - \tg_{a, \sigma^{n-1}(\omega)}(x'_{n-1}) \right|  \le C \cdot |a'-a| =\delta.
\]
Thus, to establish~\eqref{eq:d-evolution} it suffices to show that
\begin{equation}\label{eq:alt-goal}
y-x_n - j  \le M_{\sigma^{n-1}\omega} \cdot \max (0,x'_{n-1}-x_{n-1} - j)
\end{equation}
On the other hand, we have
\begin{multline*}
y-x_n - j = \tg_{a, \sigma^{n-1}(\omega)}(x'_{n-1}) - \tg_{a, \sigma^{n-1}(\omega)}(x_{n-1}) - j=  \\
= \tg_{a, \sigma^{n-1}(\omega)}(x'_{n-1}-j) -  \tg_{a, \sigma^{n-1}(\omega)}(x_{n-1}),
\end{multline*}
where the second equality uses the fact that $\tg_{a, \sigma^{n-1}(\omega)}$ commutes with integer shifts. 

\begin{figure}[h]
\includegraphics[width=0.3\textwidth]{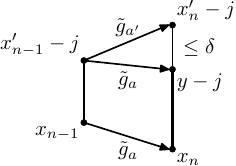}
\caption{Action of maps $\tg_{a, \sigma^{n-1}(\omega)}$ and $\tg_{a', \sigma^{n-1}(\omega)}$.  }
\end{figure}

If $x'_{n-1}-x_{n-1} - j\le 0$, then due to the monotonicity of $\tg_{a, \sigma^{n-1}(\omega)}$,
\[
y-j = \tg_{a, \sigma^{n-1}(\omega)}(x'_{n-1}-j) \le  \tg_{a, \sigma^{n-1}(\omega)}(x_{n-1}) = x_n,
\]
and~\eqref{eq:alt-goal} holds. Otherwise, we again apply the Lagrange theorem:
\[
y-x_n - j = \left. \frac{d\tg_{a, \sigma^{n-1}(\omega)}}{dx} \right|_{\xi} \cdot ((x_{n-1}'-j) - x_{n-1}) \le M_{\sigma^{n-1}\omega} \cdot \max (0,x'_{n-1}-x_{n-1} - j),
\]
where $\xi\in \mathbb{R}$ is a point between $x_{n-1}'-j$ and $x_{n-1}$. This completes the proof of Lemma \ref{l.one}.
\end{proof}
The conclusion of this lemma immediately implies the following:
\begin{coro}\label{c:steps}
Denote $d_{n,j}:= \delta+ \max(0,x'_n-x_n-j)$. Then
\[
d_{n,j} \le M_{\sigma^{n-1}\omega} \cdot d_{n-1,j}.
\]
\end{coro}

Consider now the sequence of the first moments $n_j$ when the orbit associated to~$a'$ goes $j$ full turns ahead of the one associated to $a$:
\[
n_j=\min\{n \ge 0 \mid \ x'_{n}\ge x_{n}+j\}.
\]
This sequence is well-defined for $\mu$-a.e. $\omega\in \mathfrak{M}$ due to~\eqref{e.rho} and the assumption $\rho(a')>\rho(a)$.
We have the following lemma:

\begin{lemma}\label{l.logsum}
For every $j\ge 0$
\[
\log  \frac{1}{2\delta} \le \sum_{l=n_{j}+1}^{n_{j+1}}\log M_{\sigma^{l-1}(\omega)}
\]
\end{lemma}
\begin{proof}
Applying Corollary~\ref{c:steps}, we get
\[
d_{n_{j+1},j} \le \left(\prod_{l=n_{j}+1}^{n_{j+1}} M_{\sigma^{l-1}(\omega)} \right) d_{n_j,j}.
\]
However, $d_{n_{j+1},j}\ge 1$ by definition, while $d_{n_j,j}\le 2\delta$ due to Lemma~\ref{l.one}. Taking the logarithm concludes the proof.
\end{proof}

We are now ready to complete the proof of Theorem~\ref{t.main}. Namely, for every $\mu$-regular point $\omega\in \mathfrak{M}$ we have
$$
\frac{1}{n}\sum_{l=1}^n \log M_{\sigma^{l-1}(\omega)}\to \int_{\frak{M}}\log M_{\omega'}d\mu(\omega')
$$
as $n\to \infty$. In particular, for all sufficiently large values of $n\in \mathbb{N}$ we have
$$
\frac{1}{n}\sum_{l=1}^n \log M_{\sigma^{l-1}(\omega)}<2\int_{\frak{M}}\log M_{\omega'}d\mu(\omega').
$$
Hence for all large enough $j\in \mathbb{N}$
\begin{multline*}
\frac{j}{n_j}\log \frac{1}{2\delta} \le \\ \le \frac{1}{n_j}\left(\sum_{l=1}^{n_1} \log M_{\sigma^{l-1}(\omega)}+\sum_{l=n_1+1}^{n_2} \log M_{\sigma^{l-1}(\omega)}+\ldots+ \sum_{l=n_{j-1}+1}^{n_j} \log M_{\sigma^{l-1}(\omega)}\right)=
\\
=\frac{1}{n_j}\sum_{l=1}^{n_j} \log M_{\sigma^{l-1}(\omega)}<2\int_{\frak{M}}\log M_{\omega'}d\mu(\omega').
\end{multline*}
By definition of $\{n_j\}$ we have
$
j\le |x'_{n_j}- x_{n_j}|\le j+1,
$
 and hence
 $$
 \frac{1}{n_j}|x_{n_j}'- x_{n_j}|\le \frac{2j}{n_j}\le \left(4\int_{\frak{M}}\log M_{\omega'}d\mu(\omega')\right)(\log \frac{1}{2\delta} )^{-1}
 $$
 Taking into account (\ref{e.rho}), this implies that
 \begin{equation}\label{eq:M}
 |\rho(a')-\rho(a)|\le \left(4\int_{\frak{M}}\log M_{\omega'}d\mu(\omega')\right)(\log \frac{1}{2\delta} )^{-1}.
 \end{equation}
Finally,
 \[
  \log \frac{1}{2\delta} = \log |a-a'|^{-1} - \log 2C > \frac{1}{2} \log |a-a'|^{-1}
 \]
 once $|a-a'|<e^{-4C}$, hence for such $a,a'$ the estimate~\eqref{eq:M} implies
 \[
 |\rho(a')-\rho(a)| \le R \, (\log|a'-a|^{-1})^{-1},
 \]
 where $R=8\int_{\frak{M}}\log M_{\omega'}d\mu(\omega')$.

 Due to compactness of $J$, the same inequality holds also if we remove the assumption of $a,a'$ being sufficiently close, possibly with larger value of constant $R$. This completes the proof of Theorem~\ref{t.main}.
\end{proof}

\section{The Craig-Simon Theorem on log-H\"older regularity of the IDS}\label{s.3}

In this section we explain that an application of Theorem \ref{t.main} to the Schr\"odinger cocycle of the corresponding 1D ergodic Schr\"odinger operator immediately implies that the integrated density of states must be log-H\"older regular, therefore providing a purely dynamical proof of the classical Craig-Simon result in spectral theory \cite[Theorem 5.2]{cs}. For the modern presentation of all the necessary background in the theory of ergodic Schr\"odinger operators, see \cite{DF1} and \cite{DF2}.

To define an ergodic family of discrete Schr\"odinger operators,  let us consider a homeomorphism $\sigma$ of a compact metric space $\frak{M}$,  an ergodic invariant Borel probability measure $\mu$ on $\frak{M}$, 
and a measurable function $f : \frak{M} \to \R$.
One associates a family of discrete Schr\"odinger operators on the line as follows: For $\omega \in \frak{M}$, the potential $V_\omega : \Z \to \R$ is given by $V_{\omega}(n) = f(\sigma^n \omega)$ and the operator
$H_\omega$ in $\ell^2(\Z)$ acts as
\begin{equation}\label{e.schrod}
[H_\omega \phi](n) = \phi(n+1) + \phi(n-1) + V_\omega(n) \phi(n).
\end{equation}
Since $\sigma:{(\mathfrak{M}, \mu)}\to (\mathfrak{M}, \mu)$ is ergodic, one should expect any $\sigma$-invariant measurable spectral characteristic to be almost surely constant. In particular, there is a well defined almost sure spectrum \cite{Pa}, etc. An important quantity associated with such a family of operators, $\{H_\omega\}_{\omega \in \frak{M}}$, is given by the integrated density of states, which is defined as follows; compare \cite{as, cfks} or \cite[Section 4.3]{DF1}. Define the measure $dN$ by
\begin{equation}\label{e.idsspectraldef}
\int g(\lambda) dN(\lambda) = \int \langle \delta_0 , g(H_\omega) \delta_0 \rangle \, d\mu(\omega).
\end{equation}
The integrated density of states (IDS), $N$, is then given by
\begin{equation}\label{e.idsspectraldef2}
N(E) = \int \chi_{(- \infty, E]} (\lambda) \, dN(\lambda).
\end{equation}
The terminology is explained by the fact that
\begin{equation}\label{idsform}
N(E) = \lim_{n \to \infty} \frac{\# \{ \text{eigenvalues of } H_{\omega,[1,n]} \le E \}}{n} \; \; \text{ for $\mu$-a.e.\ } \omega \in \Omega,
\end{equation}
where $H_{\omega,[1,n]}$ denotes the restriction of $H_\omega$ to the interval $[1,n]$ with Dirichlet boundary conditions. It is a basic result that the IDS is always continuous \cite{as, Pa}, \cite[Theorem 4.3.6]{DF1}; see \cite{ds} for a very short proof that also works in higher dimensions. For specific models explicit moduli of continuity can be established. For example, for the free Laplacian we have
$$
N(E) = \begin{cases} 0 & E \le -2 \\ \frac{1}{\pi} \arccos \left( - \frac{E}{2} \right) & -2 < E < 2 \\ 1 & E \ge 2, \end{cases}
$$
which is H\"older continuous. For the Anderson Model (i.e. for Schr\"odinger operators with iid random potentials) it is known that the IDS must be H\"older continuous \cite{L}, and under additional assumptions of regularity of the distribution that defines the potential stronger results are available \cite{KS, Ki, ST}. In particular, in the case of compactly supported distribution with polynomially decaying Fourier transform one can show that the IDS must be $C^\infty$, see \cite{CK}. For the Fibonacci Hamiltonian, the IDS must be H\"older continuous \cite{DG1}, while it is not always the case for operators with Sturmian potentials, see \cite{Mun} for details. H\"older continuity of the IDS in the case of quasiperiodic potentials was established in \cite{GS}.

Many spectral properties of the operator (\ref{e.schrod}), including the integrated density of states, can be expressed in terms of the corresponding Schr\"odinger cocycle. In order to define it, for each $\omega\in \mathfrak{M}$ and $E\in \mathbb{R}$ introduce the transfer matrix
$$
A_E(\omega)=\begin{pmatrix}
              E-f(\omega) & -1 \\
              1 & 0 \\
            \end{pmatrix},
$$
and consider the $SL(2, \mathbb{R})$ cocycle
$$
(\sigma, A_E):\mathfrak{M}\times \mathbb{R}^2\to \mathfrak{M}\times\mathbb{R}^2, (\omega, \bar v)\mapsto (\sigma(\omega), A_E(\omega)\bar v),
$$
usually called a Schr\"odinger cocycle corresponding to the family (\ref{e.schrod}). Each linear map $A_E(\omega)$ induces a projective map that we will denote by $g_{E, \omega}:\mathbb{RP}^1\to \mathbb{RP}^1$. It is not hard to see that in the case of bounded function $f$, if $|E|\gg 1$, then the cocycle (\ref{e.cocycle}) is uniformly hyperbolic. In this case one can choose the lifts $\tg_{E, \omega}$ in such a way that the rotation number $\rho(E)$ of the corresponding cocycle is equal to 0 for $E\gg 1$ and to $1/2$ for $E\ll -1$. For a measurable function $f$ one can choose lifts in such a way that $\rho(E)\to 0$ as $E\to +\infty$ and $\rho(E)\to 1/2$ as $E\to -\infty$. In either case, the integrated density of states $N(E)$ can be expressed via the rotation number \cite{ds83}:
$$
N(E)=1-2\rho(E).
$$
Together with Theorem \ref{t.main}, this immediately gives a purely dynamical proof of the following statement:
\begin{theorem}[Theorem 5.2 from \cite{cs}]\label{t.CS}
In the setting above, if the function $f:\mathfrak{M}\to \mathbb{R}^1$ is such that
$$
\int_{\mathfrak{M}}\log(1+ |f(\omega)|)d\mu <\infty,
$$
then the integrated density of states $N(E)$ is log-H\"older continuous, i.e. for any compact $J\subset \mathbb{R}$, for some $C>0$ and any $E_1, E_2\in J$ with $|E_1-E_2|\le 1/2$ one has
$$
| N(E_1) - N(E_2) | \le C \left( \log |E_1 - E_2|^{-1} \right)^{-1}.
$$
\end{theorem}

  A multidimensional version of this statement was provided in \cite{CS}. In both \cite{CS} and \cite{cs} the Thouless formula is used as the main tool. The Thouless formula relates the integrated density of states of an ergodic family of Schr\"odiger operators and the Lyapunov exponent of the corresponding Schr\"odinger cocycle; in the non-linear setting of Theorem \ref{t.main} one of the parts of this formula, the Lyapunov exponent, is just not defined. Notice that the original results \cite{CS} and \cite{cs} deal with bounded potentials only, but in \cite{CS} the authors make a remark that their method can be adapted to the case when the function $\log(1+|f|)$ is from $L^1$ space.

An analog of Theorem \ref{t.CS} for CMV matrices was provided in \cite{FO}. Also, in many models the regularity of Lyapunov exponent of the corresponding linear cocycle is related to the regularity of the integrated density of states, and the former was heavily studied, e.g. see   \cite {BJ}, \cite{DKP}, \cite{DK}, just to give a few examples.

The question whether Theorem \ref{t.CS} is optimal was heavily discussed in spectral theory of ergodic Schr\"odinger operators. It turned out that the modulus of continuity of the integrated density of states in general cannot be improved. It was shown that even for Anderson Model the integrated density of states does not have to be H\"older continuous with a given power \cite{Hal}. Also, \cite[Theorem 5]{Cr} essentially claims that for any continuous increasing function $N(E)$  on $[0,1]$ with $N(0)=0, N(1)=1$ that is $\alpha$-log-H\"older continuous with any $\alpha>1$, i.e. just slightly more regular than allowed by Theorem \ref{t.CS}, there exists a family of almost periodic Schr\"odinger operators with an integrated density of states given by the function $N(E)$. Finally, in \cite{KG} the limit periodic potentials were used  to show  that Theorem \ref{t.CS} is sharp, and an estimate in Theorem \ref{t.CS} cannot be replaced by any better modulus of continuity. Later examples in other special classes of ergodic Schr\"odinger operators were constructed as well. So, in \cite{DKS} an example in the class of random Schr\"odinger cocycles was constructed, and in \cite{ALSZ} -- in the class of quasiperiodic operators. In particular, these results imply that the modulus of continuity in Theorem \ref{t.main} is also optimal.

\section*{Acknowledgments}

We are grateful to Jake Fillman for multiple relevant references he provided, as well as for numerous corrections, and Grigorii Monakov for useful comments on the first draft of the paper. The first author is grateful to the American Institute of Mathematics and its support through the SQuaRE program.

\end{document}